\documentclass[12pt]{article}

\usepackage{amsmath}
\usepackage{amssymb}
\usepackage{amsthm}
\usepackage{authblk}
\usepackage{colonequals}
\usepackage[letterpaper,margin=1in]{geometry}
\usepackage{graphicx}
\usepackage{xcolor}
\usepackage{hyperref}

\definecolor{cblue}{rgb}{0.121569,0.466667,0.705882}
\definecolor{cgreen}{rgb}{0.172549,0.627451,0.172549}
\definecolor{corange}{rgb}{1.000000,0.498039,0.054902}
\hypersetup{
  linkcolor  = cblue,
  citecolor  = cgreen,
  urlcolor   = corange,
  colorlinks = true,
}

\newtheorem{theorem}{Theorem}[section]
\newtheorem{proposition}[theorem]{Proposition}
\newtheorem{conjecture}[theorem]{Conjecture}

\DeclareSymbolFont{sfoperators}{OT1}{cmss}{m}{n}
\DeclareSymbolFontAlphabet{\mathsf}{sfoperators}
\makeatletter
\renewcommand{\operator@font}{\mathgroup\symsfoperators}
\makeatother
\DeclareMathOperator{\Unif}{Unif}

\title{A revision of Litvak's conjecture on Gaussian minima and a volumetric zone conjecture}
\author{Dmitriy Kunisky}
\affil{Department of Applied Mathematics \& Statistics, Johns Hopkins University}
\date{May 3, 2026}

\newcommand{\EE}{\mathbb{E}}
\newcommand{\PP}{\mathbb{P}}
\newcommand{\RR}{\mathbb{R}}
\renewcommand{\SS}{\mathbb{S}}
\newcommand{\Px}{\mathop{\mathbb{P}}}
\newcommand{\Ex}{\mathop{\mathbb{E}}}
\newcommand{\eqd}{\stackrel{\mathrm{(d)}}{=}}
\newcommand{\N}{\mathcal N}
\newcommand{\simplex}{\Delta}
\newcommand{\coscov}{\mathrm{cos}}
\newcommand{\dd}{\,d}
\newcommand{\FejesToth}{Fejes T\'{o}th}

\begin{document}

\maketitle

\begin{abstract}
    Litvak (2018) conjectured that, for any $p > 0$, the quantity $\mathbb{E}[\min_{i = 1}^n |g_i|^p]$ where $g \sim \mathcal{N}(0, \Sigma)$ is a centered Gaussian random vector is minimized among $n \times n$ correlation matrices $\Sigma$ by the Gram matrix of the regular simplex in $\mathbb{R}^{n - 1}$.
    We disprove this conjecture: the matrix with entries $\Sigma^{\mathrm{cos}}_{ij}=\cos(\pi(i - j) / n)$ already achieves a smaller moment for $p = 2$ and $n = 4$.
    We propose that $\Sigma^{\mathrm{cos}}$ is in fact the correct minimizer of these moments for all $p > 0$ and $n \geq 1$.
    Towards proving this, we conjecture a volumetric extension of Fejes T\'{o}th's zone conjecture~(1973), whose covering version was proved by Jiang and Polyanskii (2017).
    Conditional on this conjecture, we show the stronger result that $\min_{i = 1}^n |g_i|$ for $g \sim \mathcal{N}(0, \Sigma^{\mathrm{cos}})$ is stochastically dominated by $\min_{i = 1}^n |h_i|$ for $h \sim \mathcal{N}(0, \Sigma)$ for any $n \times n$ correlation matrix $\Sigma$.
    Our counterexample $\Sigma^{\mathrm{cos}}$ was found by the AlphaEvolve AI-assisted optimization system, and we also include a brief discussion of its application to such problems.
\end{abstract}

\section{Introduction}
\pagenumbering{arabic}

Let $z \sim \N(0,I_n)$ be a standard Gaussian vector.
For an $n \times n$ correlation matrix $\Sigma$ (i.e., one having $\Sigma \succeq 0$ and $\Sigma_{ii} = 1$ for each $i = 1, \dots, n$), define the random variable
\[
  M(\Sigma) \colonequals \min_{i = 1}^n |(\Sigma^{1/2}z)_i|.
\]
Equivalently, if $g \sim \N(0,\Sigma)$, then $M(\Sigma)$ has the same law as $\min_{i = 1}^n |g_i|$.
The goal of this note is to consider, in various senses, what correlation matrix $\Sigma$ makes the random variable $M(\Sigma)$ as typically small as possible.

One natural choice, motivated by the solution (proven or conjectural) to many other such extremal questions (see below for citations), is the covariance given by the Gram matrix of a regular simplex in $\RR^{n - 1}$,
\[
  \Sigma^\simplex_{ij}
  \colonequals
  \begin{cases}
    1 & \text{if } i=j,\\
    -\frac{1}{n-1} & \text{if } i \neq j.
  \end{cases}
\]
The following conjecture claiming a specific sense in which this is true has, to the best of our knowledge, remained open since its proposal.
\begin{conjecture}[Conjecture~5.1 of \cite{Litvak-2018-SimplexMeanWidthConjecture}]
    \label{conj:litvak}
    Let $n \geq 3$.
    For every $n \times n$ correlation matrix $\Sigma$ and $p > 0$,
    \[
        \EE[M(\Sigma^\simplex)^p] \leq \EE[M(\Sigma)^p].
    \]
\end{conjecture}

The conjecture is a natural companion to the better-known simplex mean width conjecture.
That problem asks whether the regular simplex maximizes the mean width among simplices inscribed in the Euclidean ball.
In Gaussian language, it asks whether the simplex covariance maximizes among correlation matrices the expectation of the largest coordinate of a Gaussian vector, $\EE[\max_{i = 1}^n g_i]$ in the above notation.
The same note \cite{Litvak-2018-SimplexMeanWidthConjecture} also discusses many similar problems, and formulates the similar Conjecture~\ref{conj:litvak} motivated in part by considerations originating with a conjecture of Mallat--Zeitouni \cite{MZ-2011-KarhunenLoeveConjecture}.

The goal of the present note is to disprove this conjecture and propose a modification that we believe to be correct.
Our counterexample to Conjecture~\ref{conj:litvak} will use the covariance matrix
\[
  \Sigma^{\coscov}_{ij}
  \colonequals
  \cos\left(\frac{\pi(i-j)}{n}\right) \text{ for } 1 \leq i,j \leq n,
\]
a rank-two correlation matrix, which already ``beats'' $\Sigma^{\simplex}$ in the context of Conjecture~\ref{conj:litvak} for $n=4$ and $p=2$.
We visualize the matrices $\Sigma^{\coscov}$ and $\Sigma^{\simplex}$ in Figure~\ref{fig:covariances}.

\paragraph{Organization}
The rest of the note is organized as follows.
First, we prove that the above is indeed a counterexample to Conjecture~\ref{conj:litvak}~(Theorem~\ref{thm:counterexample}).
Second, we propose a strengthened version of Conjecture~\ref{conj:litvak} with $\Sigma^{\simplex}$ replaced by $\Sigma^{\coscov}$~(Conjecture~\ref{conj:main}).
Third, we propose an ancillary geometric conjecture~(Conjecture~\ref{conj:volumetric-zone}) which implies our conjecture~(Theorem~\ref{thm:conditional-domination}).
Finally, we give a brief discussion of the helpful and perhaps instructive role of artificial intelligence assistants and black-box optimization in our discovery of these results~(Section~\ref{sec:ai}).

\paragraph{Notation}
We write $\N(\mu, \Sigma)$ for the multivariate Gaussian measure with mean $\mu$ and covariance $\Sigma$.
We write $\SS^{n - 1}(r) \subset \RR^n$ for the sphere of radius $r$ in $\RR^n$, and $\SS^{n - 1} \colonequals \SS^{n - 1}(1)$ for the unit sphere.

\section{Counterexample to Conjecture~\ref{conj:litvak}}

We first discuss the structure of $\Sigma^{\coscov}$ and $M(\Sigma^{\coscov})$.
Define the vectors
\[
  v_j \colonequals \left(\cos\left(\frac{(j - 1)\pi}{n}\right), \sin\left(\frac{(j - 1)\pi}{n}\right)\right) \in \RR^2 \text{ for } 1 \leq j \leq n.
\]
These define lines whose intersection points with the unit circle in $\RR^2$ are the vertices of a regular $2n$-gon.
$\Sigma^{\coscov}$ is the Gram matrix of these vectors $v_j$, $\Sigma^{\coscov}_{ij} = \langle v_i, v_j \rangle$.
Thus, if $z \sim \N(0,I_2)$ and we define
\[
  h_k \colonequals \langle z, v_k\rangle,
\]
then $h$ has the law $\N(0,\Sigma^{\coscov})$.
The law of the associated minimum random variable therefore can be realized as
\[
  M(\Sigma^{\coscov})
  \eqd
  \min_{1 \leq j \leq n} |\langle z, v_j\rangle|.
\]

\begin{figure}[t]
  \centering
  \includegraphics[width=0.88\textwidth]{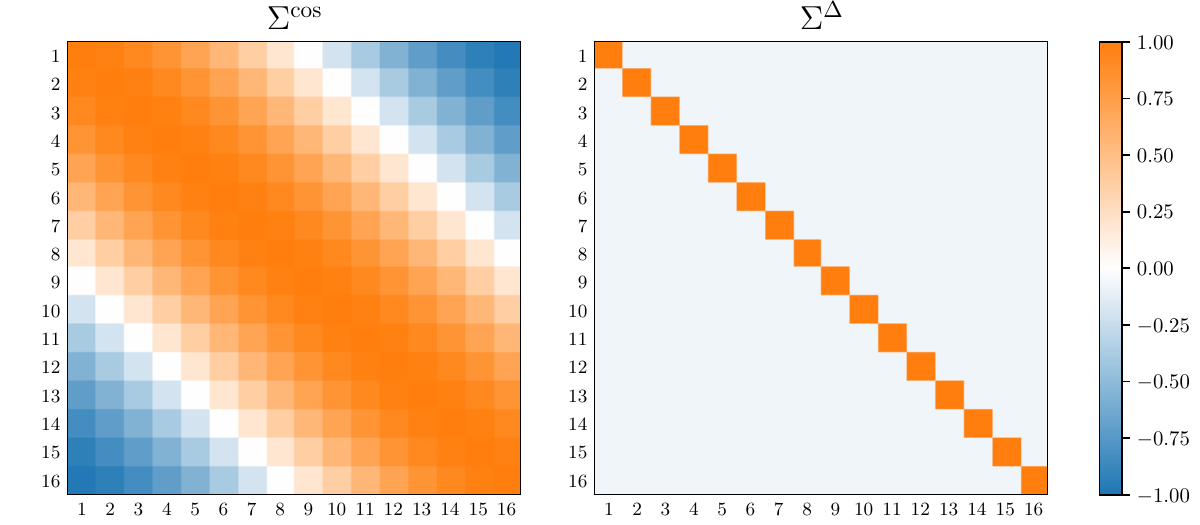}
  \caption{The cosine covariance and the simplex covariance for $n=16$, with entries shown from $-1$ in white to $1$ in black.}
  \label{fig:covariances}
\end{figure}

\begin{proposition}[Law of $M(\Sigma^{\coscov})$]
\label{prop:cosine-law}
Let $r^2 \sim \chi^2(2)$ and $\theta \sim \Unif([0,\frac{\pi}{2n}])$ be independent.\footnote{The law of such $r$ is also sometimes called the \emph{Rayleigh distribution}.}
Then, $M(\Sigma^{\coscov})$ has the same law as $r \sin(\theta)$.
Consequently, for all $t, p > 0$ we have
\begin{align*}
  \PP[M(\Sigma^{\coscov}) \geq t]
  &=
  \frac{2n}{\pi}
  \int_0^{\frac{\pi}{2n}}
    \exp\left(-\frac{t^2}{2\sin^2\theta}\right)
  \dd\theta, \\
  \EE[M(\Sigma^{\coscov})^p]
  &=
  2^{p/2}\Gamma\left(1+\frac p2\right)
  \frac{2n}{\pi}
  \int_0^{\frac{\pi}{2n}} \sin^p\theta \dd\theta.
\end{align*}
\end{proposition}

\begin{proof}
    Let $z \sim \N(0, I_2)$.
    Write this vector as $z = (r\cos\phi, r\sin\phi)$ in polar coordinates.
    We have $r^2 = z_1^2 + z_2^2$ has law $\chi^2(2)$, $\phi$ is uniformly distributed with law $\Unif([0, 2\pi])$, and the two are independent by a standard property of Gaussian random vectors.
    Since absolute values identify angles modulo $\pi$, the $n$ lines perpendicular to the $v_k$ divide the upper semicircle of the unit circle in $\RR^2$ into $n$ equal arcs.
    The angular distance $\theta$ from $\phi$ to nearest such perpendicular line therefore has law $\Unif([0, \frac{\pi}{2n}])$.
    By construction, we then have
    \[
        M(\Sigma^{\coscov}) \eqd \min_k |\langle z, v_k\rangle| = r \sin(\theta).
    \]
    The tail and moment formulas follow from the independence of $r$ and $\theta$ and explicit formulas for their densities.
\end{proof}

\begin{theorem}[Counterexample to Conjecture~\ref{conj:litvak}]
\label{thm:counterexample}
For $n = 4$, we have
\[
  \EE[M(\Sigma^{\coscov})^2]
  <
  \EE[M(\Sigma^\simplex)^2].
\]
Thus Conjecture~\ref{conj:litvak} is false for $n = 4$ and $p = 2$.
\end{theorem}
\begin{proof}
    This is simple to verify by a numerical computation, but let us give a careful rigorous justification for the sake of completeness.
By Proposition \ref{prop:cosine-law}, we may compute exactly
\[
  \EE[M(\Sigma^{\coscov})^2]
  =
  1-\frac{n\sin(\pi/n)}{\pi}.
\]
For $n=4$, this is
\[
  1-\frac{2\sqrt 2}{\pi}
  \in [0.099683, 0.099684].
\]

It remains to bound the same expectation for the simplex covariance from below.
For $n=4$, realize the simplex covariance as the Gram matrix $\Sigma^{\simplex}_{ij} = \langle u_i, u_j\rangle$ for
\[
  u_1=\frac{(1,1,1)}{\sqrt 3},\quad
  u_2=\frac{(1,-1,-1)}{\sqrt 3},\quad
  u_3=\frac{(-1,1,-1)}{\sqrt 3},\quad
  u_4=\frac{(-1,-1,1)}{\sqrt 3}.
\]
Let $g \sim \N(0, I_3)$ and write $g = (rw_1, rw_2, rw_3)$ in spherical coordinates in $\RR^3$.
Then $r$ and $w$ are independent, $\EE[r^2] = 3$, and
\[
  M(\Sigma^\simplex)
  =
  r \min_i |\langle w,u_i\rangle|.
\]
By the sign-change and coordinate-permutation symmetries of the set of lines spanned by the $u_i$, it suffices to integrate over the sector $w_1 \geq w_2 \geq w_3 \geq 0$ of the sphere $\SS^2 = \{w \in \RR^3: w_1^2+w_2^2+w_3^2=1\}$.
On this sector,
\[
  \min_i |\langle (w_1,w_2,w_3),u_i\rangle|
  =
  \frac{|w_1-w_2-w_3|}{\sqrt 3}.
\]
Using the chart for this sector
\[
  (w_1,w_2,w_3)=\frac{(1,s,st)}{\sqrt{1+s^2+s^2t^2}},
  \qquad s, t \in [0, 1]
\]
and multiplying the integral over the sector by the $2^3 \cdot 3! = 48$ congruent sectors, we get the concrete integral expression
\[
  \EE[M(\Sigma^\simplex)^2]
  =
  \frac{12}{\pi}
  \int_0^1\int_0^1
  \frac{s(1-s-st)^2}{(1+s^2+s^2t^2)^{5/2}}
  \dd s \dd t.
\]
Subdividing the unit square into $400 \times 400$ boxes and computing interval arithmetic bounds on this integral numerically gives
\[
  \EE[M(\Sigma^\simplex)^2] \in [0.139622, 0.144779].
\]
This lower bound is larger than $1-2\sqrt 2/\pi$, completing the proof.
\end{proof}

\begin{figure}[t]
  \centering
  \includegraphics[width=0.88\textwidth]{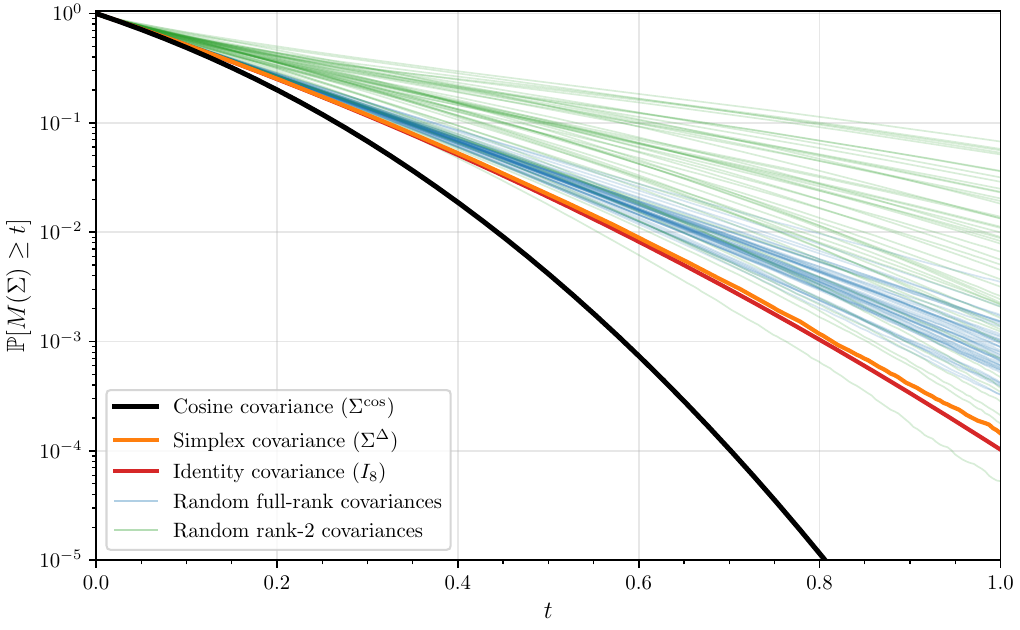}
  \caption{Upper tail probabilities of $M(\Sigma)$ for the cosine covariance, the simplex covariance, the identity covariance, 50 random full-rank covariances, and 50 random rank-two covariances with $n = 8$. Probabilities for random and simplex covariances are Monte Carlo estimates.}
  \label{fig:cdf}
\end{figure}

\section{Strengthened revision of Conjecture~\ref{conj:litvak}}

We further believe that the cosine covariance is the true optimizer in Conjecture~\ref{conj:litvak}, and that this holds in a stronger sense than moment comparison.
We propose the following revision of the Conjecture:

\begin{conjecture}
    \label{conj:main}
    For every $n \geq 1$, every correlation matrix $\Sigma$, and every $t\geq 0$,
    \[
        \PP[M(\Sigma^{\coscov})\geq t]
        \leq
        \PP[M(\Sigma)\geq t].
    \]
    That is, $M(\Sigma^{\coscov})$ is stochastically dominated by $M(\Sigma)$.
    In particular,
    \[
        \EE[M(\Sigma^{\coscov})^p]
        \leq
        \EE[M(\Sigma)^p]
    \]
    for every $p>0$.
\end{conjecture}

\noindent
This phenomenon is illustrated in Figure~\ref{fig:cdf}, which plots the tail probabilities in question for covariances $\Sigma^{\coscov}$, $\Sigma^{\simplex}$, the identity matrix, and many independent draws of random full- and low-rank correlation matrices (formed as Gram matrices of independent unit vectors of suitable dimension).
We note that the tail probability function of $\Sigma^{\coscov}$ seems to be an ``outlier'' in the space of correlation matrices, with exceptional decay much faster than either of the other fixed examples and than typical random covariances.

\section{Volumetric zone conjecture}

We have not been able to prove Conjecture~\ref{conj:main}.
However, we can reduce it to the following geometric conjecture in the spirit of the \FejesToth\ zone conjecture \cite{FejesToth-1973-ExploringPlanet}---now the Jiang--Polyanskii theorem \cite{JP-2017-ZoneConjecture}---and other similar results.
See Section~\ref{sec:related} for more references.

\begin{conjecture}[Volumetric zone conjecture]
\label{conj:volumetric-zone}
Let $d\geq 2$, let $u_1,\dots,u_n\in \SS^{d-1}$, and let $\alpha\in[0,\pi/2]$.
For a unit vector $u$, define the \emph{spherical zone}
\[
  Z(u,\alpha)
  \colonequals
  \{v \in \SS^{d-1}: |\langle v,u\rangle| \leq \sin(\alpha)\}.
\]
For $j = 1, \dots, n$, define
\[
  v_j
  \colonequals
  \left(\cos\left(\frac{(j - 1)\pi}{n}\right),\sin\left(\frac{(j - 1)\pi}{n}\right),0,\dots,0\right)
  \in \SS^{d-1}.
\]
Then
\[
  \sigma_{d-1}\left(\bigcup_{j=1}^n Z(u_j,\alpha)\right)
  \leq
  \sigma_{d-1}\left(\bigcup_{j=1}^{n} Z(v_j,\alpha)\right),
\]
where $\sigma_{d-1}$ denotes the surface measure on $\SS^{d-1}$ normalized to be a probability measure.
\end{conjecture}

The conjecture says that among $n$ equal-width zones on the sphere, the zones around $n$ evenly-spaced great subspheres sharing a subsphere of codimension two maximize the surface measure of the union.
The original zone conjecture of \cite{FejesToth-1973-ExploringPlanet}, proved by \cite{JP-2017-ZoneConjecture}, asked for the minimum width $2\alpha$ required for the union of $n$ spherical zones to cover the entire sphere, claiming that the minimum is achieved by the same configuration.
Our conjecture may be viewed as generalizing this to ``thinner'' configurations of zones that only cover a given fraction of the sphere.
We visualize the conjectural optimal configuration in Figure~\ref{fig:zones}.

\begin{figure}[t]
  \centering
  \includegraphics[width=0.5\textwidth]{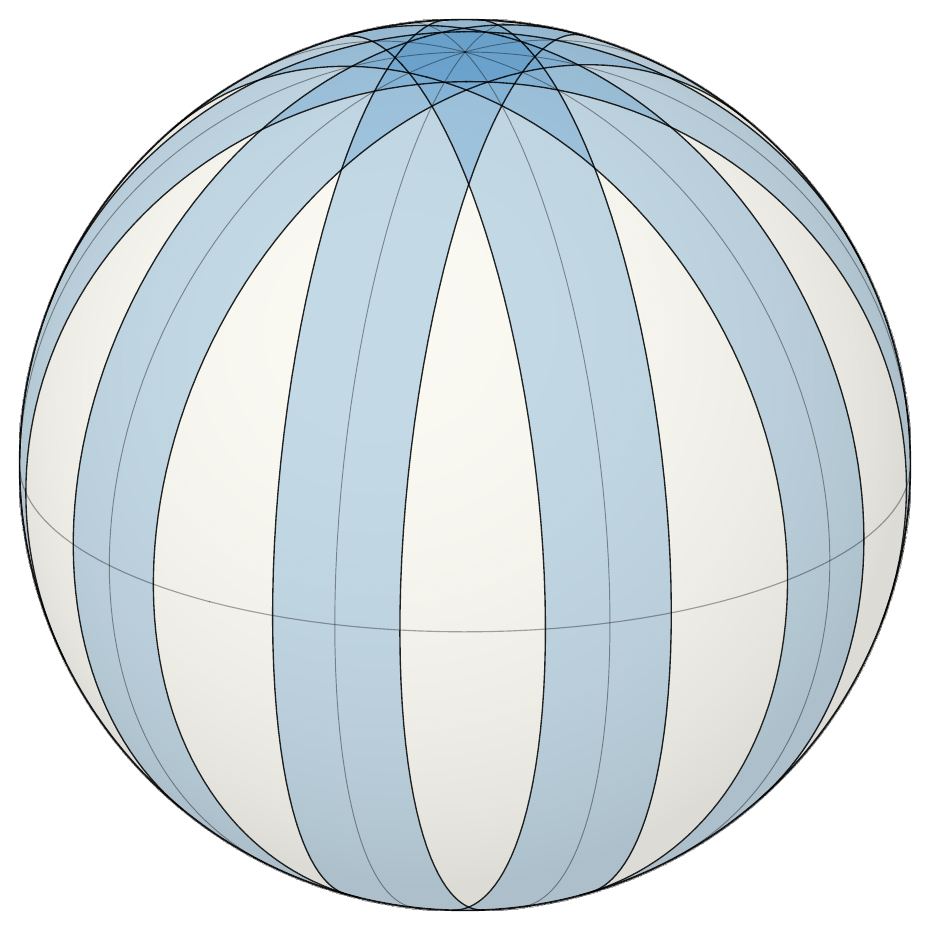}
  \caption{The optimal configuration of spherical zones maximizing the surface measure of their union proposed in Conjecture~\ref{conj:volumetric-zone}.}
  \label{fig:zones}
\end{figure}

This conjecture implies our conjecture on Gaussian minima:
\begin{theorem}
    \label{thm:conditional-domination}
    Conjecture~\ref{conj:volumetric-zone} implies Conjecture~\ref{conj:main}.
\end{theorem}
\begin{proof}
    The cases $n = 1$ and $t = 0$ of Conjecture~\ref{conj:main} are trivial, so suppose $n \geq 2$ and $t > 0$.
    Suppose Conjecture~\ref{conj:volumetric-zone} holds.
    Choose unit vectors $u_1,\dots,u_n$ in $\RR^n$ whose Gram matrix is $\Sigma$, i.e., having $\Sigma_{ij} = \langle u_i, u_j \rangle$.
    Let $g \sim \N(0,I_n)$.
    Then
    \[
        (\langle g,u_1\rangle,\dots,\langle g,u_n\rangle)
        \sim \N(0,\Sigma),
    \]
    so
    \[
        M(\Sigma)
        \eqd
        \min_{j = 1}^n |\langle g,u_j\rangle|.
    \]

    For $t>0$, define the union of slabs
    \[
        B(t)
        \colonequals
        \bigcup_{j=1}^n
        \{v\in\RR^n: |\langle v,u_j\rangle| \leq t\}.
    \]
    Then
    \[
        \PP[M(\Sigma) \leq t] = \gamma_n(B(t)),
    \]
    where $\gamma_n = \N(0, I_n)$ is the standard Gaussian measure on $\RR^n$.

    Write $S(t, r) \colonequals B(t) \cap \SS^{n - 1}(r)$, the intersection of this union of slabs with the sphere of radius $r$.
    When $t / r \geq 1$, $S(t, r) = \SS^{n - 1}(r)$.
    When $0 < t / r < 1$, we have by definition
    \[ S(t, r) = r \cdot \bigcup_{j=1}^n Z\left(u_j,\arcsin\left(\frac{t}{r}\right)\right). \]

    Thus, defining
    \[ \alpha(t, r) \colonequals \begin{cases}
     \pi/2 & \text{if } r \leq t,\\
     \arcsin(t/r) & \text{if } r > t
     \end{cases} \]
    we compute
    \begin{align*}
      \PP[M(\Sigma) \leq t]
      &= \gamma_n(B(t)) \\
      &= \Px_{g \sim \N(0, I_n)}[g \in B(t)] \\
      &= \Px_{g \sim \N(0, I_n)}[g \in S(t, \|g\|)] \\
      &= \Px_{g \sim \N(0, I_n)}\left[\frac{g}{\|g\|} \in \bigcup_{j = 1}^nZ\left(u_j,\alpha(t, \|g\|)\right)\right]
        \intertext{and since $\|g\|$ is independent of $g / \|g\|$ and the latter has law $\sigma_{n - 1}$, we have}
      &= \Ex_{g \sim \N(0, I_n)}\left[\sigma_{n - 1}\left(\bigcup_{j = 1}^nZ\left(u_j,\alpha(t, \|g\|)\right)\right)\right]
        \intertext{and letting $v_i$ be as in Conjecture~\ref{conj:volumetric-zone} and using that we assume the conjecture holds,}
      &\leq \Ex_{g \sim \N(0, I_n)}\left[\sigma_{n - 1}\left(\bigcup_{j = 1}^nZ\left(v_j,\alpha(t, \|g\|)\right)\right)\right]
        \intertext{and reversing the same manipulations}
      &= \PP[M(\Sigma^{\coscov}) \leq t],
    \end{align*}
    giving the result since the law of $M(\Sigma)$ does not have atoms for any correlation matrix $\Sigma$.
    The result on moments follows by the standard manipulation of expressing the expectations as integrals over the tail probabilities.
\end{proof}

\section{Discussion}

\subsection{Related work}
\label{sec:related}

Conjecture~\ref{conj:litvak} appears as Conjecture~5.1 in Litvak's survey \cite{Litvak-2018-SimplexMeanWidthConjecture} on the simplex mean width conjecture (which is Conjecture~1.4 of the survey).
Litvak presents it as part of a family of Gaussian extremal problems related to simplex covariances, including the mean-width problem, the simplex code conjecture, and various related geometric claims.

One motivation for Conjecture~\ref{conj:litvak} comes from a conjecture of Mallat and Zeitouni on the Karhunen--Lo{\`e}ve basis in nonlinear reconstruction \cite{MZ-2011-KarhunenLoeveConjecture}.
Litvak notes that the original intuition behind that problem would have followed from a stronger statement about minima with $p=2$ for which the identity covariance was proposed as optimal.
The same discussion records that van Handel gave an example showing that, in the three-variable case, the simplex covariance improves on the identity covariance.
While the original conjecture of Mallat--Zeitouni remains open, it was later resolved by Litvak and Tikhomirov up to absolute constants \cite{LT-2018-OrderStatisticsDependent}.

The estimates of Gordon, Litvak, Sch{\"u}tt, and Werner on minima of Gaussian and more general random variables give another nearby point of comparison \cite{GLSW-2005-MinimaGaussian,GLSW-2006-MinimumSeveral}.
Those works give sharp-order bounds for minima and order statistics under broad assumptions.
The present problem is different in that the marginal variances are fixed and the covariance matrix itself is the object to optimize.

The geometric input proposed here is a volumetric variant of \FejesToth's zone problem~\cite{FejesToth-1973-ExploringPlanet}.
Jiang and Polyanskii proved the covering version of \FejesToth's conjecture in all dimensions~\cite{JP-2017-ZoneConjecture}.
Ortega-Moreno gave a different proof for the equal-width case~\cite{OrtegaMoreno-2021-OptimalPlank}, and Zhao later streamlined that argument~\cite{Zhao-2022-ExploringPlanetRevisited}.
Glazyrin, Karasev, and Polyanskii developed polynomial plank methods which also give a strengthening for spherical segments~\cite{GKP-2023-CoveringPlanksZeros}.
For a broader view of plank problems and their applications, see Verreault's survey~\cite{Verreault-2022-PlankSurvey}.
While Conjecture~\ref{conj:volumetric-zone} seems like a natural generalization, to the best of our knowledge it has not appeared in the literature.

\subsection{AI assistance}
\label{sec:ai}

Our $\Sigma^{\coscov}$ counterexample was found by black-box minimization of the functions $f_p(\Sigma) = \EE[M(\Sigma)^p]$ using AlphaEvolve \cite{Google-2025-AlphaEvolve}, an optimization system based on code generation with large language models (see, e.g., \cite{NRT-2025-AlphaEvolveCombinatorialStructures,NRT-2026-AlphaEvolveCombinatorialStructures2} for other recent results obtained using it).
After AlphaEvolve surfaced $\Sigma^{\coscov}$ it was easy to recognize (for instance from a plot like in Figure~\ref{fig:covariances}) that there was an underlying continuous function, and then to identify it as the cosine.
Experimenting with minimizing $g_t(\Sigma) = \PP[M(\Sigma) \geq t]$ led to the stronger formulation of Conjecture~\ref{conj:main}.

Afterwards, we also considered whether more conventional optimization methods such as differential evolution and BFGS-type gradient methods could find the same optimizer.
While these often gave improvements on the simplex covariance that would just as well disprove Conjecture~\ref{conj:litvak}, they only rarely converged to (what we believe is) the true minimizer $\Sigma^{\coscov}$, especially for $n$ larger than the dimension of the smallest counterexample $n = 4$ (for $n = 3$, $M(\Sigma^{\coscov})$ and $M(\Sigma^{\simplex})$ have the same law up to a coordinate sign change) and especially when optimizing the $g_t$ functions which we evaluate by estimating the probabilities empirically.\footnote{When run over only rank-two correlation matrices, these simpler methods did often find $\Sigma^{\coscov}$; however, this is a rather unrealistic point of comparison since we did not know before seeing the result of AlphaEvolve that the optimizer should be of such low rank.}
We also found that both the $f_p$ and $g_t$ functions can have spurious local minimizers different than the global minimizer $\Sigma^{\coscov}$.
Thus, aside from convenience of implementation, AlphaEvolve seems to offer an advantage as a research tool in this setting.
In general, inspecting the output of AlphaEvolve or similar systems by hand to look for tractable structure seems to be a useful further step to translate its results into mathematical insights.
It is an intriguing question whether this process itself could be automated to further accelerate research assisted with black-box optimization.

We also remark that the connection with the \FejesToth\ zone conjecture and related literature was unknown to us until it was surfaced by GPT-5.5~Pro when asked to prove Conjecture~\ref{conj:main}.
On the other hand, this model repeatedly produced proofs which used that Conjecture as if it were included in the results of \FejesToth\ or others cited above.
Its mistake was nonetheless informative: once the unsupported step was isolated, it pointed directly to the reasonable and independently interesting Conjecture \ref{conj:volumetric-zone}.

\section*{Acknowledgments}

Thanks to the Google DeepMind AlphaEvolve team for providing early access to AlphaEvolve and in particular to Adam Zsolt Wagner for answering my many questions about using the system.

\bibliographystyle{alpha}
\bibliography{main}

\end{document}